%% file: Poly-arxiv2.tex
\documentclass[12pt]{article}
  
\usepackage{amsmath}\usepackage{amsfonts}
\usepackage{amsthm}     
\usepackage{amssymb}  
\usepackage{latexsym}
\usepackage{euscript}   
\usepackage{txfonts}  

\input{premathNEW}

\newcommand{\Int}[1]{\ensuremath{\operatorname{\mathsf{int\,}}(#1)}}

\def\emp{\varnothing}

\reversemarginpar   

\def\P{{\mcal P}}

\def\emp{\varnothing}
\begin{document}
\title{\bf Monotone Dynamical Systems with Polyhedral Order Cones and Dense Periodic Points} \author{ Morris
  W. Hirsch\\ Department of Mathematics\\ University of Wisconsin at
  Madison\\ University of California at Berkeley}
\maketitle

\begin{abstract} 

Let $X\subset 
  \R n$ be a nonempty whose interior is connected and dense in $X$.
  Assume $T\co X\approx X$ is injective, continuous, and monotone for
  the order defined by a closed convex cone $K\subset \R n$.

  \smallskip
  \noindent
{\bf Theorem.}  If $K$ is polyhedral and the set of periodic points is
dense, then $T$ 
is periodic. 

\end{abstract}

\tableofcontents

\section{Introduction}   \label{sec:intro}
 
$K\subset \R n$ (Euclidean $n$-space)  is a closed convex cone
that is {\em solid} (has
nonempty interior $\Int K$), and {\em pointed} (contains no affine
line).  Briefly: $K$ is a {\em solid order cone} in $\R n$.

$\R n$ is given the (partial) order $\succeq$ determined by $K$, referred
to as the {\em $K$-order}:
\[  y\succeq x \dimply y -x\in K, \]

$X\subset \R n$ is a  nonempty set whose interior $\Int X$ is
  connected and dense in $X$.

 $T\co X\to X$ is  an injective continuous map that is  {\em monotone} for the
  $K$-order: 
\[x \succeq y\implies Tx \succeq Ty.
\]
$x$ has {\em period $k$} if $k$ is a positive integer and
$T^kx=x$. The set of these points is denoted by $\P_k =\P_k (T)$, and
the set of periodic pointsby $\P=\P(T)=\bigcup_k\P_k$. 

$T$ is {\em periodic} if $X=\P_k$, and {\em pointwise periodic} if $X=\P$.

\smallskip
Our main concern is the following speculation:
\begin{conjecture*}  {\em If $\P$ is dense in $X$, then $T$ is periodic.}
\end{conjecture*}
\noindent We prove this for $K$ a {\em polyhedron}: the
intersection of finitely many closed affine halfspaces of $\R n$.

\begin{theorem} [\sc Main]         \label{th:main}
If $K$ is
a polyhedron and  $\P$ is dense in $X$, then  $T$ is periodic.
\end{theorem}

\medskip
The following result of {\sc D. Montgomery}
\cite{Monty}\footnote{See also {\sc S. Kaul}
  \cite {Kaul71}} is crucial to the proof:

\begin{theorem} [{\sc Montgomery}]  \label{th:monty}
  Every pointwise periodic homeomorphism of a connected manifold  is
  periodic.
\end{theorem}
\noindent This implies a sharper version of
Theorem \ref{th:main} for analytic maps:
\begin{theorem}               \label{th:cor1}
If $K$ is a polyhedron and $T$ is analytic but not periodic, $\P$ is
nowhere dense.
\end{theorem}

Two observations simplify Theorem \ref{th:main}:

\begin{itemize}
\item 
 The hypothesis holds if $ \P$ is dense in $\Int X$, because $\Int X$ is
dense in $X$. 

\item  The conclusion holds if $\Int X \subset \P$. 

For then Montgomery's Theorem implies $\Int X\subset \P_k$, hence
$X\subset \P_k$ because $\Int X$ is dense in $X$ and $\P_k$ is closed
in $X$,
\end{itemize}

\subsection{Notation}
 $i,j,k,l$ always denote  positive integers, and
$a, b,  p, q, u, v,  x, y, z$ points of $\R n$.

 $x\preceq y$ is a synonym for $y \succeq x$.  If $x\preceq y$ and $x\ne
y$ we write $x\prec$ or $y\succ x$. 

The relations $x\ll y$ and $y\gg x$ mean $ y-x \in \Int K$.

A set $S$ is {\em totally ordered} if  $x, y\in S \implies x \preceq y$ or
$x\succeq y$.

If $x\preceq y$, the {\em order interval} 
$[x, y]$ is $\{z\co x\preceq z \preceq y\} = K_x \cap{-K_y}$.

The translation of $K$ by $x\in \R n$ is $K_x:=\{w+x, w\in K.\}$

The image of a set or point $\xi$ under a map $H$ is denoted by $H\xi$
or  $H(\xi)$. A set $S$ is {\em positively invariant} under $H$ if $HS
\subset S$, {\em invariant} if $H\xi=\xi$, and {\em periodically
  invariant} if $H^k\xi=\xi$.

\section{Proof of the Main Theorem }   \label{sec:proofmain}
We derive three topological consequences, valid
even if $K$ is not polyhedral, from the standing
assumptions that $T$ is monotone and $\P$ is dense:  

\begin{proposition}             \label{th:propA}
Assune $p,q \in \P_k$ are such that
\[  p\ll q, \qquad p, q\in \P_k. \qquad [p,q]\subset X.
\]
Then $T^k([p,q]=[p,q]$.
\end{proposition}
\begin{proof} 
It suffices to take $k=1$. Evidently $T\P=\P$, and
$T[p,q]\subset[p,q]$ because $T$ is monotone, whence $\Int{ [p,
    q]}\cap \P$ is positively invariant under $T$.  The conclusion
follows because $ \int{[p, q]}\cap \P$ is dense in $[p,q]$ and $T$ is
continuous.
\end{proof}

\begin{proposition}             \label{th:arcs}
Assume $u, v\in \P_k, u\ll v$, and $ [u,v]\subset X$.  There is a compact arc
$J\subset \P_k\cap [u,v]$ joining $u$ to $v$ that is totally ordered by
$\ll$.\footnote{This result is adapted from {\sc Hirsch \& Smith}
  \cite{Hirsch-Smith05}, Theorems 5,11 \& 5,15.}
\end{proposition}
\begin{proof}
An application of  Zorn's Lemma yields  a maximal  set 
$J\subset [u, v]\cap P$ 
 such that:  $J$ is totally ordered by $\ll, \ u=\max J, \ v=\min J$.
Maximality implies $J$ is compact and connected and $u,v\in J$, so $J$
is an arc ({\sc Wilder} \cite{Wilder}, Theorem I.11.23).
\end{proof}


\begin{proposition}           \label{th:base}
Let  $M\subset X$ be a topological manifold of dimension $n-1$ without
boundary,  closed in $X$.  

\begin{description}

\item[(i)] $\P$ is dense in $M$.

\item[(ii)] If $M$ is  periodically invariant, it has a neighborhood
  base $\mcal W$ of periodically invariant open subsets. 

\item[(iii)] $\P$ is dense in each $W\in\mcal W$.

\end{description}
\end{proposition}
\begin{proof} Lefschetz duality ({\sc Spanier} \cite{Spanier66}) shows
  that every point of $M$ has arbitrarily small open ball
  neighborhoods $W\subset X$ separated by $M$ into two disjoint open
  sets $U$ and $V$:
\begin{equation}                \label{eq:wmuv}
W\,\verb=\=\,M = U\cup V, \quad U\cap V=\emp.
\end{equation}
Since  $\P$ is dense in $W$ there exist  $u, v\in \P\cap U, v\in
\P\cap V$ such that
\[u\ll v, \qquad [u,v]\subset W.
\]
There is an arc $J\subset \P\cap [u, v]$ joining $u$ and $v$
(Proposition \ref{th:arcs}).  Because $J$ is separated by $M$, there
exists $p\in J\cap M$, implying $p\in W\cap \P$.  This proves (i), 
assertion (ii) follows because $M\cap \Int{[u,v]}$ is periodically
invariant, and (iii) holds because $\P$ is dense $M$ and $W$ is open
in $M$.
\end{proof} 

Let $\mcal T (m)$ stand for the statement of Theorem \ref{th:main}
for the case $n=m$.  Then $\mcal T (0)$ is trivial, and we use the
following inductive hypothesis:
 
\begin{hypothesis} [{\sc Induction}]         \label{th:hypind}
$n \ge 1$ and $\mcal T (n-1)$ holds.
\end{hypothesis}
Let $Q\subset \R n$ be a compact $n$-dimensional polyhedral cell, such
as $[p,q]$ with $p\ll q$.  Its
boundary $\p Q$ is the union of finitely many convex compact $(n-1)$-cells,  the
{\em faces} of $Q$.  Each face $F$ is  the intersection of $\p[p,q]$
with a unique affine 
hyperplane $E^{n-1}$.  The corresponding {\em open face} is
$F^\circ:=F\,\verb=\=\,\p F$, an open $(n-1)$-cell in $E^{n-1}$.
Distinct
open faces are disjoint, and their union is dense and open in $\p Q$.
\begin{proposition}             \label{th:propF}      
Assume $p, q\in \P_k, \ p\ll q, \ [p, q]\subset X$.  Then $T|\p[p,q]$ is
periodic.
\end{proposition}
\begin{proof}
The $(n-1)$-dimensional compact polyhedron $\p[p,q]$ is invariant
under $T^k$ (Proposition \ref{th:propA}). 
Each open face $F^\circ\subset \p[p,q]$ is a union of periodically
invariant open subsets $W_\lam$, each an $(n-1)$-manifold without
boundary, with $\P$ dense in $W_\lam$ (Proposition \ref{th:base}).

The maps $T|W_\lam$ are periodic by the Induction Hypothesis, so
$F^{\circ} \subset \P$.
Montgomery's Theorem implies $T|F^{\circ}$ is periodic, so $T|F$ is
also periodic.  Since $\p[p, q]$ is the union of the finitely many
faces $F$, it follows that $T|\p[p, q]$ is periodic.
\end{proof}

To complete the inductive proof of the Main Theorem, it suffices by
Montgomery's theorem to prove that an arbitrary $x\in X$ is periodic.
 As $X$ is open in $\R n$ and $\P$ is dense in $X$,
there is an order interval $[a, b]\subset X$ such that
\[
a\ll x \ll b, \qquad a, b \in \P_k.
\]
By Proposition \ref{th:arcs}, $a$ and $b$ are the endpoints of a
compact arc $J\subset \P_k\cap [a,b]$, totally ordered by $\ll$.
Define $p, q\in J$:
\[
p:=\sup\,\{y\in J\co y\preceq x\}, \qquad q:=\inf\,\{y\in J\co y\succeq
  x\}.
\]
If $p=q=x$ then $x\in \P_k$.  Otherwise $p\ll q$ and $x\in\p[p,q]$, whence $x\in \P$
by Proposition \ref{th:propF}.  \qed

\end{document}

%% file: premathNEW.tex
\newcommand{\mcal}[1]{{\mathcal {#1}}} 

\newtheorem{theorem}{Theorem}  

\newtheorem{proposition}[theorem]{Proposition}
\newtheorem{hypothesis}[theorem]{Hypothesis}

\theoremstyle{definition}
\theoremstyle{definition}
\theoremstyle{definition}
\theoremstyle{definition}

\theoremstyle{definition}
\theoremstyle{definition}

\theoremstyle{plain}\newtheorem*{theorem*}{Theorem}
\theoremstyle{plain}\newtheorem*{corollary*}{Corollary}
\theoremstyle{plain}\newtheorem*{lemma*}{Lemma}
\theoremstyle{plain}\newtheorem*{proposition*}{Proposition}
\theoremstyle{remark}\newtheorem*{remark*}{Remark}
\theoremstyle{remark}\newtheorem*{remarks*}{Remarks}
\theoremstyle{definition}\newtheorem*{conjecture*}{Conjecture}
\theoremstyle{definition}\newtheorem*{definition*}{Definition}
\theoremstyle{definition}\newtheorem*{example*}{Example}

\theoremstyle{definition}\newtheorem*{question*}{Question}
\theoremstyle{definition}\newtheorem*{questions*}{Questions}

\theoremstyle{definition}\newtheorem*{hypothesis*}{Hypothesis}

\def\qed{\ifhmode\unskip\nobreak\fi\ifmmode\ifinner\else\hskip5pt\fi\fi
 \hfill\hbox{\hskip5pt\vrule width4pt height6pt depth1.5pt\hskip1pt}}
 
\renewcommand{\int}{\ensuremath{\operatorname{\mathsf{int}}}}

\newcommand{\dimply}{\ensuremath{\:\Longleftrightarrow\:}}

\newcommand{\R}[1]{\ensuremath{{\mathbb R}^{#1}}} 

\renewcommand{\P}[1]{\ensuremath{{\bf P}^{#1}}} 

\newcommand{\emp}{\ensuremath{\emptyset}}

\def\d#1dt{\frac{d#1}{dt}}    

\newcommand{\lam}{\ensuremath{\lambda}}

\newcommand{\co}{\colon\thinspace} 

\newcommand{\p}{\ensuremath{\partial}}